\documentclass[12pt, reqno]{amsart}

\usepackage{amsfonts,latexsym,amsthm,amssymb,amsmath,amscd,euscript,bm}
\usepackage[sc]{mathpazo}
\usepackage[margin = 2cm]{geometry}
\usepackage{enumitem}
\usepackage{hyperref}

\usepackage{color}
\usepackage{hyperref}
\usepackage{url}
\usepackage{breakurl}
\newcommand{\bburl}[1]{\textcolor{blue}{\url{#1}}}

\numberwithin{equation}{section}
\numberwithin{part}{section}

\newtheorem{theorem}{Theorem}
\numberwithin{theorem}{subsection} 

\newtheorem{lemma}[theorem]{Lemma}

\theoremstyle{definition}

\theoremstyle{remark}
\newtheorem*{remark}{Remark}

\newcommand{\nc}{\newcommand}
\nc{\on}[1]{\operatorname{#1}}

\nc{\R}{\mathbb R}
\nc{\C}{\mathbb C}
\nc{\Q}{\mathbb Q}
\nc{\Z}{\mathbb Z}
\nc{\N}{\mathbb N}
\nc{\HH}{\mathbb H}
\nc{\DD}{\mathbb D}
\nc{\TT}{\mathbb T}
\nc{\EE}{\mathbb E}

\nc{\cT}{\mathcal T}
\nc{\cP}{\mathcal P}
\nc{\cM}{\mathcal M}
\nc{\cC}{\mathcal C}
\nc{\cB}{\mathcal B}
\nc{\cG}{\mathcal G}
\nc{\cA}{\mathcal A}
\nc{\cS}{\mathcal S}
\nc{\cF}{\mathcal F}
\nc{\cL}{\mathcal L}
\nc{\cR}{\mathcal R}

\nc{\frakI}{\mathfrak I}

\nc{\diam}{\operatorname{diam}}     
\nc{\osc}{\operatorname{osc}}       
\nc{\inter}{\mathrm{o}}             
\nc{\close}[1]{\overline{#1}}       
\nc{\supp}{\operatorname{supp}}     
\nc{\Prob}{\operatorname{Pr}}       

\nc{\Symp}{\mathsf{Sp}}
\nc{\SpOrthO}{\mathsf{SO(odd)}}
\nc{\SpOrthE}{\mathsf{SO(even)}}
\nc{\Orth}{\mathsf O}
\nc{\Unit}{\mathsf U}
\nc{\UnitSp}{\mathsf{USp}}

\renewcommand{\epsilon}{\varepsilon}

\title
{
	\textsc{Determining optimal test functions for $2$-level densities}
}

\author[Bo\l dyriew]{El\.zbieta Bo\l dyriew}
\email{\textcolor{blue}{\href{mailto:ela.boldyriew@gmail.com}{ela.boldyriew@gmail.com}}}

\author[Chen]{Fangu Chen}
\email{\textcolor{blue}{\href{mailto:fangu@berkeley.edu}{fangu@berkeley.edu}}}
\address{Department of Mathematics, University of California Berkeley, Berkeley, CA 94720}

\author[Devlin]{Charles Devlin VI}
\email{\textcolor{blue}{\href{mailto:cpd6@uchicago.edu}{cpd6@uchicago.edu}}}
\address{Department of Mathematics, University of Chicago, Chicago, IL 60637}

\author[Miller]{Steven J. Miller}
\email{\textcolor{blue}{\href{mailto:sjm1@williams.edu}{sjm1@williams.edu}},  \textcolor{blue}{\href{Steven.Miller.MC.96@aya.yale.edu}{Steven.Miller.MC.96@aya.yale.edu}}}
\address{Department of Mathematics and Statistics, Williams College, Williamstown, MA 01267}

\author[Zhao]{Jason Zhao}
\email{\textcolor{blue}{\href{mailto:zhao.j@berkeley.edu}{zhao.j@berkeley.edu}}}
\address{Department of Mathematics, University of California, Berkeley, CA 94720}

\date{\today}

\subjclass[2010]{11Mxx (primary); 45Bxx (secondary)}

\keywords{Random matrix theory, $L$-functions, low-lying zeros, optimal test functions, Fredholm theory}

\thanks{This research was supported by NSF grant DMS1947438 and Williams College. The fourth listed author was supported by NSF grant DMS1561945. We also thank John Haviland, Fernando Trejos Su\'arez and Jiahui Yu for their comments on the problem during our many fruitful conversations, as well as the referee for their helpful suggestions. No data sets were used in this analysis; however, some Mathematica code was written to estimate certain integrals, and is available upon request by emailing S. J. Miller.}

\begin{document}


\maketitle

\thispagestyle{empty}

\begin{abstract}
Katz and Sarnak conjectured a correspondence between the $n$-level density statistics of zeros from families of $L$-functions with eigenvalues from random matrix ensembles. In many cases the sums of smooth test functions, whose Fourier transforms are finitely supported, over scaled zeros in a family converge to an integral of the test function against a density $W_{n, G}$ depending on the symmetry $G$ of the family (unitary, symplectic or orthogonal). This integral bounds the average order of vanishing at the central point of the corresponding family of $L$-functions.

We can obtain better estimates on this vanishing in two ways. The first is to do more number theory, and prove results for larger $n$ and greater support; the second is to do functional analysis and obtain better test functions to minimize the resulting integrals. We pursue the latter here when $n=2$, minimizing 	\[ \frac{1}{\Phi(0, 0)} \int_{{\mathbb R}^2} W_{2,G} (x, y) \Phi(x, y) dx dy \] over test functions $\Phi : {\mathbb R}^2 \to [0, \infty)$ with compactly supported Fourier transform. We study a
restricted version of this optimization problem, imposing that our test functions take the form $\phi(x) \psi(y)$ for some fixed admissible $\psi(y)$ and
$\supp{\widehat \phi} \subseteq [-1, 1]$. Extending results from the 1-level case, namely the functional analytic arguments of Iwaniec, Luo and Sarnak and the differential equations method introduced by Freeman and Miller, we explicitly solve for the optimal $\phi$ for appropriately chosen fixed test function $\psi$. The solution allows us to deduce strong estimates for the proportion of newforms of rank $0$ or $2$ in the case of $\SpOrthE$, rank $1$ or $3$ in the case of $\SpOrthO$, and rank at most $2$ for $\Orth$, $\Symp$, and $\Unit$; our estimates are a significant strengthening of the best known estimates obtained with the $1$-level density. As a representative example, the previous best 1-level analysis yields a lower bound of 0.7839 for vanishing to order at most 2 for the $\SpOrthE$ family of cuspidal newforms, and our 2-level work improves this to 0.952694. We conclude by discussing further improvements on estimates by the method of iteration.
\end{abstract}

\tableofcontents

\section{Introduction}

\subsection{Background}

The location and distribution of zeros of $L$-functions play a central role in numerous number theory problems. In many situations, the more one knows about their spacing, the stronger results one has. For example, the fact that the Riemann zeta function and Dirichlet $L$-functions do not vanish on the line ${\Re}(s) = 1$ yields the Prime Number Theorem, and that for a given modulus $q$ each arithmetic progression that can contain infinitely many primes does so, and to first order they all have the same number of primes up to $x$. If the Generalized Riemann Hypothesis for Dirichlet $L$-functions is true, we can improve the error terms in these counts to of size $x^{1/2 + \epsilon}$ for any $\epsilon$ (with a little work we can replace $x^\epsilon$ with a power of $\log x$); see for example  \cite{Da,IK}.

Next, if the positive imaginary parts of the nontrivial zeros of all Dirichlet $L$-functions with primitive Dirichlet characters are linearly independent over $\Q$ (the Grand Simplicity Hypothesis), then Rubinstein and Sarnak \cite{RubSa} proved Chebyshev's Bias, which quantifies how often up to $x$ each ordering of the possible residue classes modulo $q$ occurs. Assuming GRH, the non-trivial zeros of $L$-functions lie on the line ${\rm Re}(s) = 1/2$. There are many models that predict their behavior, especially Random Matrix Theory which states that families of $L$-functions as the conductors tend to infinity are modeled by ensembles of matrices (unitary, orthogonal and symplectic) with size tending to infinity; see in particular \cite{BFMT-B, Con, FM, Ha} (survey articles) and \cite{KatzSarnak, KatzSarnak2, KeSn1, KeSn2, KeSn3} (connections between number theory and random matrix theory), \cite{bogo, BogoKeat} (lower order terms), \cite{DHKMS1, DHKMS2} (lowest zero), \cite{CFKRS} (moments). Some of the most studied statistics are the $n$-level correlations and spacings between zeros \cite{Gallagher, Hej, Mon, Od1, Od2}, $n$-level densities \cite{FiM, GJMMPP, HR2, LM, Mil3, OS1, OS2, Rub, RudnickSarnak} (Dirichlet forms), \cite{Gao, Gu, HM, ILS, Mil4, RR, Ro} (cuspidal newforms), \cite{Mil1, Mil2, MilMo,  Yo1, Yo2} (elliptic curves), \cite{AAILMZ, AM} (Maass forms), \cite{FI, MilPe, Ya} (number fields), \cite{GK} (GL(3) families), and general \cite{CS, HR1} and compound families \cite{DM1, DM2, ShTe} (one could also look at moments and central values, and the models' predictions fare well here; for other approaches see \cite{CFZ1, CFZ2, GHK}). There are many reasons for interest in what happens on the critical line, such as the existence of many small gaps (relative to the average spacing) implying bounds for the class number problem \cite{CI}, or the famous Birch and Swinnerton-Dyer Conjecture equating the order of vanishing of elliptic curve $L$-functions at the central point to the rank of the Mordell-Weil group of rational points \cite{BSD1, BSD2, Gol}.

The last is the starting point for our investigations: given a family of $L$-functions ordered by conductor, what fraction with conductor at most a fixed size (tending to infinity) vanish to a given order at the central point? As remarked, for elliptic curves this is conjecturally related to the group of rational solutions. In this paper we use the $n$-level density of Katz-Sarnak \cite{KatzSarnak, KatzSarnak2}, applied to an even, non-negative Schwartz test function to obtain upper bounds. This approach was pioneered in \cite{ILS} (see in particular Appendix A), and extended further for the 1-level density in by Freeman \cite{FreemanThesis, FreemanMiller}.

There are many advantages to studying the $n$-level density in general, and the 2-level (which is our focus) in particular; we define these in the next subsection, and just state the applications. First, as Miller showed in his thesis \cite{Mil1}, while the three orthogonal groups have indistinguishable support for test functions whose Fourier transforms are supported in $(-1,1)$, the 2-level densities are distinguishable from each other (and the symplectic and unitary cases) for arbitrarily small support. Next, the $n$-level density yields results of the following form (see \cite[Corollary~1.9]{HM}): there are constants $c_n$ and $r_n$ such that as $N \to \infty$ through primes, the proportion of weight $k$ cuspidal newforms of level $N$ and sign $+1$ (resp. $-1$) having a zero of some given order $r \geq r_n$ at the central point is at most $c_n/r^n$; equivalently, the proportion with fewer than $r$ zeros at the central point is at least $1 - c_n / r^n$. Unfortunately as $n$ increases in practice the support where we can prove results decreases, and thus the constants $c_n$ grow with $n$ and the results are initially \emph{worse} for small $r$, though eventually the greater decay kicks in and better results than those from the 1-level are obtained. We summarize some recent progress, which examined consequences of using the optimal 2-level results from this paper for $n=4$, in Appendix \ref{sec:applimiller}; see \cite{LiM} for full details.

Below we provide details on just how results on the distribution of zeros near the central point (the $n$-level density) translate to bounds on the order of vanishing. The focus of our work is to provide the best possible upper bounds. This leads to a functional analysis problem, involving the optimization of integrals involving the even Schwartz test function.

\subsection{$n$-level density}

Let $\cF$ be a family of cuspidal newforms, and to each $f \in \cF$ we associate the $L$-function
    \[ L(s, f) \ = \  \sum_{n = 1}^\infty \frac{a_{n, f}}{n^s}. \]
We assume that the Riemann hypothesis holds for each $L(s, f)$ and for all Dirichlet $L$-functions, that is, we can enumerate the non-trivial zeros of $L(s, f)$ by
    \[ \rho^{(j)}_f \ = \ \frac12 + i \gamma_f^{(j)} \]
for $\gamma_f^{(j)} \in \R$ increasingly ordered and centered about zero (i.e. so that $-\gamma_f^{(j)} = \gamma_f^{(-j)}$ for all $j$). A standard argument (see \cite{IK}) shows that the number of zeros with $|\gamma_f^{(j)}|$ bounded by an absolute large constant is of order $\log c_f$ for some constant $c_f > 1$ known as the \emph{analytic conductor}.
 It is of interest to study the statistics of these ``low-lying'' zeros of $L(s, f)$, and to this end Katz and Sarnak \cite{KatzSarnak} introduced the \emph{$n$-level density}
\begin{equation}
    D_n (f; \Phi) \ := \  \sum_{\substack{j_1, \dots, j_n \\ j_i \neq \pm j_k}} \Phi \left( \frac{\log c_f}{2\pi} \gamma_f^{(j_1)}, \dots, \frac{\log c_f}{2\pi} \gamma_f^{(j_n)} \right)	\label{def:density}
\end{equation}
for \emph{test functions} $\Phi: \R^n \to \R$, which we take to be non-negative\footnote{We only need the test function to be non-negative if we wish to obtain bounds on the order of vanishing. As the zeros are symmetrically distributed on the critical line, there is no loss in taking $\Phi$ to be even.} even Schwartz class functions with compactly supported Fourier transform and $\Phi(0) > 0$. In practice the sum (\ref{def:density}) is impossible to evaluate asymptotically, since by choice of $\Phi$ it essentially captures only a bounded number of zeros. Instead we study averages over finite subfamilies $\cF (Q) := \{ f \in \cF : c_f \leq Q \}$, namely
\begin{equation}
    \EE (D_n (f; \Phi), Q) \ := \  \frac{1}{\# \cF (Q)} \sum_{f \in \cF (Q)} D(f; \Phi). \label{eq:averages}
\end{equation}
If $\cF$ is a complete family of cuspidal newforms in a spectral sense, there exists a distribution $W_{n, \cF}$ such that
\begin{equation}
    \lim_{Q \to \infty} \EE (D_n (f; \Phi), Q) \ = \   \frac{1}{\Phi(0)}\int_{\R^n} \Phi(x_1, \dots, x_n) W_{n, \cF} (x_1, \dots, x_n) dx_1 \cdots dx_n.
\end{equation}
Katz and Sarnak \cite{KatzSarnak, KatzSarnak2} conjectured that $W_{n, \cF}$ depends on a corresponding symmetry group $G(\cF)$, the scaling limit of one of the classical compact groups, so for the remainder we shall write $W_{n, G}$ in place of $W_{n, \cF}$.

Define
	\[ K(y) \ := \  \frac{\sin (\pi y)}{\pi y}, \qquad K_\epsilon (x, y) \ := \  K(x - y) + \epsilon K(x + y) \]
for $\epsilon = 0, \pm 1$. The corresponding $n$-level densities have the following distinct closed form determinant expansions \cite{HughesMiller, KatzSarnak}:
	\begin{align}
		W_{n, \SpOrthE} (x) 	
			&\ = \  \det \left( K_1 (x_i, x_j) \right)_{i, j \leq n}, \label{eq:nlevelSOeven} \\
		W_{n, \SpOrthO} (x)
			&\ = \  \det \left( K_{-1} (x_i, x_j) \right)_{i, j, \leq n} + \sum_{k  =  1}^n \delta (x_k) \det \left( K_{-1} (x_i, x_j) \right)_{i, j, \neq k},  \\
		W_{n, \Orth} (x)
			&\ = \  \frac12 W_{n, \SpOrthE} (x) + \frac12 W_{n, \SpOrthO} (x), \\
		W_{n, \Unit} (x)
			&\ = \  \det \left( K_0 (x_i, x_j) \right)_{i, j, \leq n}, \\
		W_{n, \Symp} (x)			
			&\ = \ \det \left( K_{-1} (x_i, x_j) \right)_{i, j, \leq n} \label{eq:nlevelSymp}.
	\end{align}

\subsection{Main result}

It is discussed in \cite{FreemanMiller} and \cite{ILS} that the 1-level density gives estimates on the average order of vanishing of $L$-functions at the central point in a family. Here we deal with the 2-level densities, which has the advantage of giving better estimates on higher vanishing at the central point. Let $\Prob (m)$ denote the limit as $Q \to \infty$ of the proportion of $f \in \cF(Q)$ with $r_f = m$, where $r_f$ is the order of the zero of $L(s, f)$ at $s = 1/2$. Considering (\ref{eq:averages}) for $n = 2$ and taking only terms with $\gamma_f^{(j_1)} = \gamma_f^{(j_2)} = 0$ gives the bound
\begin{align}
    4\sum_{m = 1}^\infty \left(m (m - 1) \Prob (2m) + m^2 \Prob(2m + 1) \right) \ \leq \ \frac{1}{\Phi(0,0)} \int_{\R^2} \Phi(x, y) W_{2, G} (x, y) dx dy. \label{eq:boundorder}
\end{align}
It is therefore of interest to choose $\Phi$ optimally to obtain the best bound on the left-hand side of (\ref{eq:boundorder}). Rather than minimizing over test functions of two variables, we instead fix a single variable test function $\psi$ and, imposing the restriction $\Phi(x, y) = \phi(x) \psi (y)$, minimize over single variable test functions $\phi$ with $\supp \widehat \phi \subseteq [-1, 1]$. Doing so greatly simplifies the calculations and still leads to good bounds; with more work this assumption can be relaxed, see for example \cite{LiM}. For our fixed $\psi$, we consider
	\begin{equation}
		\psi(y) \ = \  \left( \frac{\sin (\pi y)}{\pi y} \right)^2.\label{eq:fixedtest}
	\end{equation}	
It follows from Corollary A.2 in Appendix A of \cite{ILS} that the optimal test functions with Fourier transforms supported in $[-1, 1]$ for the 1-level densities are exactly scalar multiples of $\psi$, making it the natural choice of fixed test function. In addition, as we note at the start of Section \ref{sec:quadker}, the optimization problem for $\phi$ given our choice \eqref{eq:fixedtest} admits a particularly nice solution, another reason why \eqref{eq:fixedtest} is an ideal choice. Our main result is to solve this restricted optimization problem for $\psi$ as defined above.




\begin{theorem}
	Let $\psi$ be as in (\ref{eq:fixedtest}). For each of the classical compact groups $G = \SpOrthE$, $\SpOrthO$, $\Unit$, $\Orth$, and $\Symp$, there exists an optimal square integrable function $g_{G, \psi} \in L^2 [-1/2, 1/2]$ and constant $c_{G, \psi}$ such that
		\begin{equation}
			\frac{c_{G, \psi}}{\int_{-1/2}^{1/2} g(x) dx} \ = \  \inf_\phi  \frac{1}{\phi(0) \psi(0)} \int_{\R^2} \phi(x) \psi(y) W_{2, G} (x, y) dx dy, \label{eq:thmidentity} \end{equation}	
	where the infimum is taken over test functions $\phi$ with Fourier transform satisfying $\supp \widehat \phi \subseteq [-1, 1]$. The constants and optimal square integrable functions are given by
		\begin{equation}
			c_{G, \psi} \ = \
				\begin{cases}
					\frac12, 		&\text{if } G = \Symp,\\
					1, 				&\text{if } G \ = \  \Unit, \\
					\frac32,			&\text{if } G \ = \  \SpOrthE, \SpOrthO, \Orth,
				\end{cases}
		\end{equation}
	and
		\begin{align}
			g_{\SpOrthE, \psi} (x)
				&\ = \  \frac{216 \cos(4x/\sqrt 3) + 36 \sqrt 3 \sin(2/\sqrt 3)}{162 \cos(2/\sqrt 3) - 5 \sqrt{3} \sin(2/\sqrt 3)}, \\
			g_{\SpOrthO, \psi} (x) &\ = \  \frac{8 \cos (4 x/\sqrt{3})+12 \sqrt{3} \sin(2/\sqrt{3})}{11 \sqrt{3} \sin (2/\sqrt{3})+2 \cos(2/\sqrt{3})}, \\
			g_{\Unit, \psi} (x)
				&\ = \  \frac{6 \cos (2x) + 6 \sin(1)}{3 \cos(1) + 4 \sin(1)}, \\
			g_{\Orth, \psi} (x)
				&\ = \   \frac{36\cos(4x/\sqrt{3})+18\sqrt{3} \sin(2/\sqrt{3})}{18 \cos(2/\sqrt{3}) + 13 \sqrt{3} \sin(2/\sqrt{3})}, \\
			g_{\Symp, \psi} (x)
				&\ = \  \frac{8 \cos (4x) + 12 \sin(2)}{2 \cos (2) + 3 \sin(2)}.
		\end{align}	
	Additionally, the optimal test function $\phi_{G, \psi}$ realizing the infimum in (\ref{eq:thmidentity}) satisfies $\widehat{\phi_{G, \psi}} = g_{G, \psi} * g_{G, \psi}$.
		\label{thm:optimalgpsi}
\end{theorem}
The test function $\psi$ is used in Section 1 of \cite{ILS} to obtain naive bounds on the average order of vanishing. Similarly, we can compute naive bounds for the 2-level densities by taking $\Phi (x, y) = \psi(x) \psi(y)$. Table \ref{tab:bounds} shows that the bounds derived from Theorem \ref{thm:optimalgpsi} significantly improve the naive bounds.

\begin{table}[h]
	\begin{center}
		\begin{tabular}{|l||c|c|}
			\hline
			\textbf{Family} & \textbf{Naive bounds} & \textbf{Closed form of (\ref{eq:thmidentity})} \\ \hline
			$\SpOrthE$ & $\frac{5}{12} \approx 0.416666$ & $\frac{1}{96} \left(54 \sqrt{3} \cot \left(\frac{2}{\sqrt{3}}\right)-5\right) \approx 0.378448$ \\\hline
			$\SpOrthO$ & $\frac{13}{12} \approx 1.083333$ & $\frac{1}{32} \left(33+2 \sqrt{3} \cot \left(\frac{2}{\sqrt{3}}\right)\right) \approx 1.07909$ \\
			\hline
			$\Orth$ & $\frac34 \approx 0.75$ & $\frac{1}{24} \left(13+6 \sqrt{3} \cot \left(\frac{2}{\sqrt{3}}\right)\right)\approx 0.733014 $\\\hline
			$\Unit$ & $\frac12 \approx 0.5$ & $\frac{1}{12} (4+3 \cot (1)) \approx 0.493856$\\\hline
			$\Symp$ & $\frac{1}{12} \approx 0.083333$ & $\frac{1}{32} (3+2 \cot (2)) \approx 0.0651464$\\\hline
		\end{tabular}
		\caption{Comparing naive bounds taking $\phi = \psi$ with the optimal value over support in $[-1, 1]$ from (\ref{eq:thmidentity}) for each of the classical compact groups.} \label{tab:bounds}
	\end{center}
\end{table}



\subsection{Applications to vanishing at the central point}
If we instead consider (\ref{eq:averages}) for $n = 1$ and again take only the terms at the central point, we obtain the $1$-level analogue to (\ref{eq:boundorder}),
\begin{equation}
	\sum_{m = 1}^\infty m \Prob(m) \ \leq \ \frac{1}{\phi(0)} \int_\R \phi(x) W_{1, G} (x) dx dy \label{eq:boundorder1}.
\end{equation}
Comparing with (\ref{eq:boundorder}), we see that the $2$-level densities gives improvements on estimates for the average order of vanishing $m$ by a factor of $m$. For example, consider the orthogonal groups $\SpOrthE$, where the order at the central point is always even, and $\SpOrthO$, where the order is always odd. It was shown in Appendix A of \cite{ILS} that the optimal values for the right-hand side of (\ref{eq:boundorder1}) for test functions with Fourier transforms supported in $[-2, 2]$ are
\begin{equation}
	\inf_\phi \frac{1}{\phi(0)} \int_\R \phi(x) W_{1, G} (x) dx
		\ = \
		\begin{cases}
			\frac18 \left(3 + \cot\left(\frac14\right)\right) \approx 0.8645, 		&\text{if } G = \SpOrthE, \\
			\frac18 \left(5 + \cot\left(\frac14\right)\right) \approx 1.1145, 		&\text{if } G = \SpOrthO.
		\end{cases}	\label{eq:bounds1}
\end{equation}
Note that our bounds in Table \ref{tab:bounds} are better on the same order of magnitude as those above. Thus for $\SpOrthE$ and $\SpOrthO$ we immediately see improvements on upper bounds for order two and above. For higher orders such as $m = 2020$  or $2021$, we see significant improvements. We obtain the $1$-level bounds applying (\ref{eq:bounds1}) to (\ref{eq:boundorder1}), yielding
\begin{align}
	\Prob(2020) &\ \lessapprox \ 4.280 \cdot 10^{-4} \qquad \text{if } G = \SpOrthE,\\
	\Prob(2021) &\ \lessapprox \ 5.515 \cdot 10^{-4} \qquad \text{if } G = \SpOrthO.
\end{align}
The $2$-level bounds are obtained using the optimal values listed in Table \ref{tab:bounds} and (\ref{eq:boundorder}), yielding
\begin{align}
	\Prob(2020) &\ \lessapprox \ 9.284\cdot 10^{-8} \qquad \text{if } G = \SpOrthE,\\
	\Prob(2021) &\ \lessapprox \ 2.645 \cdot 10^{-7} \qquad \text{if } G = \SpOrthO.
\end{align}
Moreover, subtracting either (\ref{eq:boundorder}) or (\ref{eq:boundorder1}) from $\sum_m \Prob(m) = 1$ gives lower bounds on low orders of vanishing. In the case of the groups $\SpOrthE$ and $\SpOrthO$, we can also use the parity of the order for marginally better results. For $\SpOrthE$, the $1$-level and $2$-level lower bounds are respectively
\begin{equation}
	\sum_{m = 0}^k \Prob(2m) \ \geq \  1 -  \frac{1}{(2k + 2)\phi(0)} \int_\R \phi(x) W_{1, \SpOrthE} (x) dx, \label{eq:lower1even}
\end{equation}
and
\begin{equation}
	\sum_{m = 0}^k \Prob(2m) \ \geq \  1 -  \frac{1}{4 k(k + 1)\Phi(0, 0)} \int_{\R^2} \Phi(x, y) W_{2, \SpOrthE} (x, y) dx dy, \label{eq:lower2even}
\end{equation}
For $\SpOrthO$, the $1$-level and $2$-level lower bounds are respectively
\begin{equation}
	\sum_{m = 0}^k \Prob(2m + 1) \ \geq \  1 -  \frac{1}{(2k + 3)\phi(0)} \int_\R \phi(x) W_{1, \SpOrthO} (x) dx, \label{eq:lower1odd}
\end{equation}
and
\begin{equation}
	\sum_{m = 0}^k \Prob(2m + 1) \ \geq \  1 -  \frac{1}{4 (k + 1)^2\Phi(0, 0)} \int_{\R^2} \Phi(x, y) W_{2, \SpOrthO} (x, y) dx dy. \label{eq:lower2odd}
\end{equation}
For example, consider $k = 1$. Using the 1-level estimates (\ref{eq:lower1even}) and (\ref{eq:lower1odd}) with the values from (\ref{eq:bounds1}) yields the lower bounds
\begin{align}
	\Prob(0) + \Prob(2) &\gtrapprox 0.7839 \qquad \text{if } G = \SpOrthE,\\
	\Prob(1) + \Prob(3) &\gtrapprox 0.7771 \qquad \text{if } G = \SpOrthO.
\end{align}
On the other hand, using the 2-level estimates (\ref{eq:lower2even}) and (\ref{eq:lower2odd}) and the optimal values from Table \ref{tab:bounds} yields the bounds
\begin{align}
    \Prob(0) + \Prob(2) &\gtrapprox 0.952694 \qquad \text{if } G = \SpOrthE, \label{eq:lower2evenbound} \\
    \Prob(1) + \Prob(3) &\gtrapprox 0.932556 \qquad \text{if } G = \SpOrthO. \label{eq:lower2oddbound}
\end{align}
Similarly, $2$-level estimates give
\begin{align}
    \Prob(0) + \Prob(1) + \Prob(2) &\gtrapprox 0.816746 \qquad \text{if } G = \Orth, \label{eq:lower2orthobound} \\
    \Prob(0) + \Prob(1) + \Prob(2) &\gtrapprox 0.876536 \qquad \text{if } G = \Unit, \label{eq:lower2unitbound} \\
    \Prob(0) + \Prob(1) + \Prob(2) &\gtrapprox 0.983713 \qquad \text{if } G = \Symp. \label{eq:lower2unitspbound}
\end{align}

We see that using the $2$-level estimates also provides significant improvements on the lower bound for low average order of vanishing at the central point compared to the $1$-level estimates from \cite{ILS} despite considering smaller support. A natural question to examine is how large of a support needs to be considered for the $1$-level estimates to provide better bounds on these low orders of vanishing than our $2$-level estimates for fixed support in $[-1, 1]$. Currently the largest support where the optimal test functions for the $1$-level densities are known is $[-3, 3]$, shown in \cite{FreemanMiller, FreemanThesis}. The corresponding optimal values follow from Corollary 1.2 of \cite{FreemanMiller} taking $\sigma = 1.5$, which yields
\begin{equation}
	\inf_\phi \frac{1}{\phi(0)} \int_\R \phi(x) W_{1, G} (x) dx
		\approx
		\begin{cases}
			0.60363, 		&\text{if } G = \SpOrthE, \\
			1.04304, 		&\text{if } G = \SpOrthO,
		\end{cases}	\label{eq:jessebounds}
\end{equation}
where the infimum is taken over test functions with Fourier transforms supported in $[-3, 3]$. Applying the optimal values from (\ref{eq:jessebounds}) to the $1$-level estimates (\ref{eq:lower1even}) and (\ref{eq:lower1odd}), we obtain
\begin{align}
	\Prob(0) + \Prob(2) &\gtrapprox 0.84909 \qquad \text{if } G = \SpOrthE,\\
	\Prob(1) + \Prob(3) &\gtrapprox 0.79139 \qquad \text{if } G = \SpOrthO.
\end{align}
These $1$-level bounds are still worse than our $2$-level bounds from (\ref{eq:lower2evenbound}) and (\ref{eq:lower2oddbound}), so it would seem that we need much larger support for the $1$-level estimate to surpass the $2$-level estimates for fixed support at these low orders of vanishing. Further, to date there are no families where the 1-level density has been done for support as large as $[-2, 2]$; the best we have is up to $[-2,
2]$ (or a little more for cuspidal newforms if we assume Hypothesis S from
\cite{ILS}); however, we do have the 2-level up to $[-1,1]$.
This motivates further research into deriving optimal test functions for higher level densities and small support.

\section{Proof of Theorem \ref{thm:optimalgpsi}}

\subsection{Functional analytic setup}

Prior literature on the optimization problem, such as \cite{FreemanMiller} and \cite{FreemanThesis}, dealt with the 1-level densities following the functional analytic approach outlined in Appendix A of \cite{ILS}. We want to impose restrictions so that such an approach is amenable to the 2-level density optimization problem. To that end, we consider the optimization over test functions of the form $\Phi(x, y) = \phi(x) \psi(y)$ for fixed admissible $\psi(y)$ with $\supp \widehat\psi \subseteq [-1, 1]$. This reduces the problem to one analogous to the 1-level density, where we are optimizing over one-variable test functions. Explicitly, we want to compute
\begin{equation}
    \inf_\phi \frac{1}{\phi(0) \psi(0)} \int_{\R^2} \phi(x) \psi(y) W_{2, G} (x, y) dx dy \label{eq:2minimize}
\end{equation}
where the infimum is taken over test functions $\phi : \R \to \R$ with $\supp \widehat \phi \subseteq [-1, 1]$. Attacking the optimization problem via the Fourier transform is more promising than a direct approach. On the transform side, assumptions on the support reduce an integration over the entire plane $\R^2$ to an integration over the square $[-1, 1] \times [-1, 1]$, and the 2-level densities themselves are unwieldy to work with, while their Fourier transforms are sums of linear polynomials in $|x|$ and Dirac delta functions. Moreover, Gallagher \cite{Gallagher} noted that a correspondence exists between admissible test functions $\phi$ and square-integrable functions. Namely, it follows by the Ahiezer and Paley-Wiener theorems that $\phi$ is a test function with $\supp \widehat \phi \subseteq [-1, 1]$ if and only if there exists $f \in L^2 [-1/2, 1/2]$ such that
	\begin{equation}
		\widehat \phi (x) \ = \  (f * \breve{f} ) (x),
	\end{equation}	
where
	\begin{equation}
		\breve{f} (x) \ = \  \close{f( - x)}.
	\end{equation}	
Thus rather than minimizing a functional over test functions, we can instead view the problem as minimizing a functional $\tilde R_{G, \psi}$ on a subset of $L^2 [-1/2, 1/2]$, defined by
	\begin{equation}
		\tilde R_{G, \psi} (f) \ := \  \frac{1}{\phi(0) \psi(0)} \int_{\R^2} \phi(x) \psi(y) W_{2, G} (x, y) dx dy. \label{eq:originalfunctional}
	\end{equation}
This perspective gives access to more functional analytic tools, namely Fredholm theory. Motivated by our earlier remarks on the Fourier transform, we apply the Plancharel theorem to write
	\begin{equation}
		\tilde R_{G, \psi} (f) \ = \  \frac{1}{\phi(0)} \int_{-1}^1 \widehat \phi (x) \tilde V_{G, \psi} (x) dx,
	\end{equation}	
where we have a weight function $\tilde V_{G, \psi}$ given by
	\begin{equation}
		\tilde V_{G, \psi} (x) \ = \  \frac{1}{\psi(0)}\int_{-1}^1 \widehat \psi (y) \widehat{W_{2, G}} (x, y) dy. \label{eq:vtilde}
	\end{equation}	
In the 1-level case, the role of the weight function is played by the Fourier transforms of the 1-level distributions (\ref{eq:nlevelSOeven}) - (\ref{eq:nlevelSymp}), which take the form $\delta + m_G$. Analogously, following calculations due to Hughes and Miller \cite{HughesMiller}, for each of the classical compact groups the weight function (\ref{eq:vtilde}) takes the form
	\begin{equation}
		\tilde V_{G, \psi} (x) \ = \   c_{G, \psi} \delta(x) + \tilde m_{G, \psi} (x) \mathbb 1_{[-1, 1]} (x) \label{eq:vtilde2}
	\end{equation}	
for constants $c_{G, \psi} \in \R$ and kernel $\tilde m_{G, \psi} \in L^2 [-1, 1]$ depending on our choice of initial test function $\psi$ and the classical compact group $G$, namely
	\begin{equation}
		c_{G, \psi} \ = \  \frac{\widehat \psi (0)}{\psi (0)} +
			\begin{cases}
				- \frac12, 		&\text{if } G = \Symp,\\
				0, 				&\text{if } G = \Unit, \\
				\frac12,			&\text{if } G = \SpOrthE, \SpOrthO, \Orth,
			\end{cases}
		\end{equation}
and
\begin{align}
	\tilde m_{\SpOrthE, \psi} (x)
	    &\ = \ \frac{1}{2}\left(\frac{\widehat{\psi}(0)}{\psi(0)} + \frac{1}{2}\right) + 2 \frac{\widehat{\psi}(x)}{\psi(0)}(|x|-1)-\frac{\int_{|x|-1}^{1-|x|}\widehat{\psi}(y) dy}{\psi(0)},\\
	\tilde  m_{\SpOrthO, \psi} (x)
        &\ = \  \frac{1}{2}\left(\frac{\widehat{\psi}(0)}{\psi(0)} - \frac{3}{2}\right) + 2 \frac{\widehat{\psi}(x)}{\psi(0)}(|x|-1)+\frac{\int_{|x|-1}^{1-|x|}\widehat{\psi}(y) dy}{\psi(0)},\\
	\tilde m_{\Orth, \Psi}(x)
		&\ = \  \frac{1}{2}\left(\frac{\widehat{\psi}(0)}{\psi(0)} - \frac{1}{2}\right) + 2 \frac{\widehat{\psi}(x)}{\psi(0)}(|x|-1),\\
	\tilde m_{\Unit, \psi} (x)
	    &\ = \ 	\frac{\widehat{\psi}(x)}{\psi(0)}(|x|-1),\\
	\tilde m_{\Symp, \psi} (x)
	    &\ = \ -\frac{1}{2}\left(\frac{\widehat{\psi}(0)}{\psi(0)} - \frac{1}{2}\right) + 2 \frac{\widehat{\psi}(x)}{\psi(0)}(|x|-1)+\frac{\int_{|x|-1}^{1-|x|}\widehat{\psi}(y) dy}{\psi(0)}.		
\end{align}
To complete the analogy with the 1-level case, we consider the normalized functional, weight function and kernel:
	\begin{align}
		R_{G, \psi} (f) &\ := \  \frac{\tilde R_{G, \psi}(f)}{c_{G, \psi}}, \label{eq:quadraticform} \\
		V_{G, \psi} (x) &\ := \  \frac{\tilde V_{G, \psi} (x)}{c_{G, \psi}}, \\
		m_{G, \psi} (x) &\ := \  \frac{\tilde m_{G, \psi} (x)}{c_{G, \psi}}.
	\end{align}	
Consider the compact self-adjoint operator $K_{G, \psi} :L^2 [-1/2, 1/2] \to L^2 [-1/2, 1/2]$ (see Appendix A of \cite{ILS}) defined by
	\begin{equation}
		(K_{G, \psi} f) (x) \ := \  \int_{-1/2}^{1/2} m_{G, \psi} (x - y) f(y) dy.
	\end{equation}	
Then the optimization problem (\ref{eq:2minimize}) is equivalent to the minimization of the quadratic form
    \begin{equation}
    	R_{G, \psi} (f) \ = \  \frac{\langle (I + K_{G, \psi})f, f \rangle}{|\langle f, 1 \rangle|^2}
    \end{equation}	
subject to the linear constraint $\langle f, 1 \rangle \neq 0$,{where
    \begin{equation}
        \langle f, g \rangle \ := \ \int_{-1/2}^{1/2} f(y) g(y) \, \mathrm{d} y.
    \end{equation}
We know $R_{G, \psi} \geq 0$, which implies $I + K_{G, \psi}$ is positive definite. If $I + K_{G, \psi}$ is non-singular, that is, $-1$ is not an eigenvalue of $K_{G, \psi}$, then it follows from the Fredholm alternative that there exists a unique $g_{G, \psi} \in L^2 [-1/2, 1/2]$ satisfying the integral equation
	\begin{equation}
			(I + K_{G, \psi}) (g_{G, \psi}) (x) \ = \  g_{G, \psi} (x) + \int_{-1/2}^{1/2} m_{G, \psi} (x - y) g_{G, \psi} (y) dy \ = \  1 \label{eq:fredholm}
	\end{equation}	
for $x \in [-1/2, 1/2]$, and by Proposition A.1 of \cite{ILS},
	\begin{equation}
		\inf_f R_{G, \psi} (f) \ = \  \frac{1}{\langle 1, g_{G, \psi} \rangle}. \label{eq:ILSminimum}
	\end{equation}	
Observe (\ref{eq:thmidentity})	in Theorem \ref{thm:optimalgpsi} follows directly from (\ref{eq:originalfunctional}), (\ref{eq:quadraticform}) and (\ref{eq:ILSminimum}). It remains to check the following lemma.

\begin{lemma}
The operator $I + K_{G, \psi}$ is non-singular, that is, $-1$ is not an eigenvalue of $K_{G, \psi}$. \label{lem:nonsingular}
\end{lemma}

\begin{proof}
	Let $f \in L^2 [-1/2, 1/2]$ such that $K_{G, \psi} f = -f$, i.e.,
		\[ 0 \ = \  f(x) + \int_{-1/2}^{1/2} m_{G, \psi} (x - y) f(y) dy. \]
	This is a Fredholm equation of the second kind, and the unique continuous solution is given by the corresponding Liouville-Neumann series, which, in this case, is the constant zero function. On the other hand, $m_{G, \psi}$ is uniformly continuous on $[-1, 1]$, so for fixed $\epsilon > 0$, there exists $\delta > 0$ witnessing the uniform continuity. Then
		\[ |f(x + h) - f(x)|\ \leq\ \int_{-1/2}^{1/2} |m_{G, \psi} (x + h - y) - m_{G, \psi}(x - y)| |f(y)| dy \ \leq\ \epsilon ||f||_{L^2} \]
	whenever $|h| < \delta$. It follows that $f$ is continuous, so by uniqueness, $f = 0$. 	
\end{proof}

\begin{remark}
	A similar argument shows that the solution $g_{G, \psi}$ to (\ref{eq:fredholm}) is continuous.
\end{remark}

\subsection{Solving a Fredholm integral equation with quadratic kernel} \label{sec:quadker}

Let $\psi$ be as in (\ref{eq:fixedtest}), which has Fourier transform
	\begin{equation}	
		\widehat \psi (x) \ = \  (\mathbb 1_{[-1/2, 1/2]} * \mathbb 1_{[-1/2, 1/2]})(x) \ = \  (1 - |x|) \mathbb 1_{[-1, 1]} (x).
	\end{equation}	
Not only is this the natural choice for the fixed test function $\psi$, it also lends to an elegant derivation of the corresponding optimal $g_{G, \psi}$, as the kernels $m_{G, \psi}$ take the form of quadratic polynomials in $|x|$ on the interval $[-1, 1]$. Namely,
\begin{align}
	m_{\SpOrthE, \psi} (x)
		&\ = \  -\frac32 + \frac83 |x| - \frac23 x^2, \\
	m_{\SpOrthO, \psi} (x)
		&\ = \  -\frac{5}{6} + \frac83 |x| - 2 x^2, \\
	m_{\Orth, \psi} (x)
		&\ = \  -\frac76 + \frac83 |x| -\frac43 2x^2, \\
	m_{\Unit, \psi} (x)
		&\ = \  -1 + 2 |x| - x^2,\\
	m_{\Symp, \psi} (x)
		&\ = \  - \frac52 + 8|x| - 6 x^2.
\end{align}
Prior experience with the analogous 1-level problem in \cite{ILS}, \cite{FreemanThesis} and \cite{FreemanMiller} suggests that $g_{G, \psi}$ takes the form of an even trigonometric polynomial. Indeed, not only does this hold in Theorem \ref{thm:optimalgpsi}, this holds in generality for Fredholm integral equations where the kernel is an even quadratic polynomial in $|x|$, as Lemma \ref{lem:nonsingular} relies only on uniform continuity of the kernel.

\begin{theorem}
	Let $a, b, c \in \R$ with $b \geq 0$. The following Fredholm integral equation with quadratic kernel,
		\begin{equation}
			1 \ = \  g(x) + \int_{-1/2}^{1/2} (a + b |x - y| + c |x - y|^2) g(y) dy, 	\label{eq:fredholmquad}
		\end{equation}
	admits the unique continuous solution
		\begin{equation}
			g(x) \ = \  \frac{6b^{3/2} (b + c) \cos(\sqrt{2b} x) - 6 \sqrt{2} b c \sin(\sqrt{b/2})}{6 \sqrt{b} (b + c)^2 \cos(\sqrt{b/2}) + \sqrt{2} (6 a b^2 + 3b^3 + 3b^2 c + b c (c - 12) - 6c^2) \sin (\sqrt{b/2})}. \label{eq:quadsol}
		\end{equation}\label{thm:quadratic}	
\end{theorem}

Theorem \ref{thm:optimalgpsi} follows as an immediate corollary. We prove Theorem \ref{thm:quadratic} following a differential equations argument due to Freeman and Miller. Observe that the left-hand side of (\ref{eq:fredholmquad}) is constant, so derivatives of the expression on the right vanish. Assuming $g$ is smooth, we differentiate under the integral sign to obtain
\begin{align*}
		\frac{d}{dx}  \int_{-1/2}^{1/2} |x - y| g(y) dy
			&\ = \ \frac{d}{dx} \left( \int_{-1/2}^x (x - y) g(y) dy +  \int_x^{1/2} (y - x) g(y) dy \right)\\
			&\ = \  \int_{-1/2}^x g(y) dy - \int_x^{1/2} g(y) dy, \\
		\frac{d^2}{dx^2}  \int_{-1/2}^{1/2} |x - y| g(y) dy
			&\ = \  2g(x),
	\end{align*}
and
	\begin{align*}
		\frac{d}{dx} \int_{-1/2}^{1/2} (x - y)^2 g(y) dy
			&\ = \  \int_{-1/2}^{1/2} (2x - 2y) g(y) dy, \\
		\frac{d^2}{dx^2} \int_{-1/2}^{1/2} (x - y)^2 g(y) dy
			&\ = \  2 \int_{-1/2}^{1/2} g(y) dy.
	\end{align*}
We thereby obtain the corresponding system of linear homogeneous differential equations,
	\begin{align}
		1
			&\ = \  g(0) + a \int_{-1/2}^{1/2} |y| g(y) dy, \label{eq:diff1}\\
		0
			&\ = \  g'' (x) + 2b g(x) + 2c \int_{-1/2}^{1/2} g(y) dy, \label{eq:diff2} \\
		0
			&\ = \  g''' (x) + 2b g'(x).		\label{eq:diff3}	
	\end{align}
Equation (\ref{eq:diff1}) is exactly (\ref{eq:fredholmquad}) taking $x = 0$. We obtain (\ref{eq:diff2}) and (\ref{eq:diff3}) by differentiating (\ref{eq:fredholmquad}) under the integral sign twice and thrice respectively. Assuming $g$ is even, solutions to (\ref{eq:diff3}) take the form $g(x) = A \cos(\sqrt{2b} x) + C$ for some constants $A, C \in \R$. Substituting into (\ref{eq:diff2}) reduces these two degrees of freedom to one,
	\begin{align}
		0
			&\ = \  g'' (x) + 2b g(x) + 2c \int_{-1/2}^{1/2} g(y) dy \ = \  2C (b + c) + \frac{4A c}{\sqrt{2b}} \sin\left( \frac{\sqrt{2b}}{2} \right).
	\end{align}
	This shows that
	\begin{equation}
		 g(x) \ = \  A\cos(\sqrt{2b} x) - A\frac{2c}{b + c} \frac{\sin\left( \sqrt{b/2} \right)}{\sqrt{2b}}.
	\end{equation}
	Substituting the above into (\ref{eq:diff1}) allows us to solve for $A$ explicitly, which completes the derivation of (\ref{eq:quadsol}). \hfill $\Box$

\begin{remark}
    The same differential equations method can be applied for kernels which take the form of of higher order polynomials in $|x|$ on $[-1, 1]$, that is, $m(x) = p(x)$ for some degree $n$ polynomial $p$. Future approaches to the optimization problem may want to consider optimizing over fixed test functions $\psi$ which produce kernels of such form.
\end{remark}

\section{Iteration}

One can improve the bounds found in Table \ref{tab:bounds} by replacing $\psi$ by $\phi_{G, \psi}$ and optimizing accordingly. As $\widehat{\phi_{G, \psi}}$ takes the form of a piecewise trigonometric polynomial for each of the classical compact groups, the methods used in Section \ref{sec:quadker} are not applicable. We instead appeal to the standard approach to solving Fredholm integral equations by the method of iteration. Suppose an even continuous kernel $m: [-1, 1] \to \R$ satisfies
	\begin{equation}
		\int_{-1/2}^{1/2} \int_{-1/2}^{1/2} |m(x - y)|^2 dx dy < 1.
	\end{equation}
Define a self-adjoint compact operator $K: L^2 [-1/2, 1/2] \to L^2[-1/2, 1/2]$ by
	\begin{equation}
		(Kf)(x) \ := \  \int_{-1/2}^{1/2} m(x - y) f(y) dy.
	\end{equation}
It follows from the Cauchy-Schwarz inequality that the operator norm satisfies $||K||_{L^2 \to L^2} < 1$, that is, $K$ is a contraction mapping, since
	\begin{equation}
		||Kg||_2^2 \ = \  \int_{-1/2}^{1/2} \left( \int_{-1/2}^{1/2} m(x - y) g(y) dy \right)^2 dx \leq ||g||_2^2 \int_{-1/2}^{1/2} \int_{-1/2}^{1/2} |m(x - y)|^2 dx dy < ||g||_2^2.
	\end{equation}	
Thus by the Weierstrass $M$-test, the series
	\begin{equation}
		g(x) \ := \  \sum_{n = 0}^\infty (-1)^n K^n (1)(x) \label{eq:series}
	\end{equation}
converges absolutely and uniformly on the interval $[-1/2, 1/2]$. Moreover, it is the unique continuous solution to the Fredholm integral equation $(I + K) g = 1$. Since the series converges absolutely, we can integrate term by term to obtain the corresponding minimum value,
	\begin{equation}
		\frac{c_{G, \phi}}{\langle 1, g \rangle} \ = \  c_{G, \phi}\left( \sum_{n = 0}^\infty (-1)^n \int_{-1/2}^{1/2} K^n (1) (x) dx \right)^{-1}. \label{eq:seriesmin}
	\end{equation}
Unfortunately, this method of deriving new bounds is computationally intensive as we need to compute $n$-dimensional integrals of unwieldy expressions. Additionally, depending on the rate of convergence we may need to compute a large number of terms to obtain meaningful degrees of accuracy. For the purposes of this paper we focus on the unitary group, where these challenges can be avoided.

For brevity, denote $\phi := \phi_{\Unit, \psi}$. In this case we know the series converges, as
\begin{align}
		\int_{-1/2}^{1/2} \int_{-1/2}^{1/2} |m_{\Unit, \phi}(x - y)|^2 dx dy \ = \  \frac{2 \sin^2 (1) (128 - 110 \cos (2) - 37 \sin (2))}{3 (-8 + 6 \cos (2) - \sin (2))^2} < 1.
\end{align}
Moreover, $\widehat{\phi}$ is non-negative, so it follows that $(-1)^n K^n_{\Unit, \phi}(1)$ is non-negative. We can therefore truncate the series in (\ref{eq:seriesmin}) to obtain an upper bound, as the terms are non-negative. Summing five terms gives the bound
	\begin{equation}
		\inf_\Phi \frac{1}{\Phi(0, 0)} \int_{\R^2} \Phi(x, y) W_{2, \Unit} (x, y) dx dy\ \leq\ c_{G, \psi}\left( \sum_{n = 0}^5 (-1)^n \int_{-1/2}^{1/2} K^n (1) (x) dx \right)^{-1} \approx 0.4888,
	\end{equation}
a small improvement on our previous bound in Table \ref{tab:bounds}.

\section{Future Work}

We conclude with some remarks on how to further improve estimates on the optimal value of right hand side of (\ref{eq:boundorder}) and thereby obtaining improved bounds on the average order of vanishing at the central point. Recall that the quantity of interest is
\begin{equation}
	 \frac{1}{\phi(0) \psi(0)} \int_{\R^2} \phi(x) \psi(y) W_{2, \cF} (x, y) dx dy \label{eq:further}
\end{equation}
for one variable test functions $\phi$ and $\psi$ with Fourier transforms supported in $[-1, 1]$. As hinted in the previous section, so far we have optimized over test functions $\phi$ for fixed $\psi$, so a natural approach to improving the bounds is to iterate: for each $k \geq 1$ we find optimal $\phi_{k + 1}$ for fixed $\phi_{k}$, then taking $\phi = \phi_k$ and $\psi = \phi_{k + 1}$ in (\ref{eq:further}) gives a non-increasing sequence in $k$. The hope is that this iteration converges to a global minimum. The main challenge encountered in computing the iteration is that we lack a general method of obtaining a closed form expression for the optimal $\phi_k$. From prior experience in \cite{ILS, FM} and Theorem \ref{thm:optimalgpsi}, we expect the series (\ref{eq:series}) converges to a piecewise continuous trigonometric polynomial, and likewise if we continue to iterate the optimization. Another approach to the problem may therefore be a numerical search of the space of piecewise continuous trigonometric polynomials.

We plan on generalizing these arguments to arbitrary $n$-level densities, that is, optimizing over $n$-variables test function taking the form $\Phi(x_1, \dots, x_n) = \phi_1 (x_1) \cdots \phi_n (x_n)$ for fixed $\phi_1, \dots, \phi_{n - 1}$. Again, the natural choices for the fixed test functions are
	\[ \phi_1 (x)\ =\ \cdots\ =\ \phi_{n - 1} (x)\ =\ \left( \frac{\sin (\pi x)}{\pi x} \right)^2, \]
and we suspect the differential equations method outlined in Section \ref{sec:quadker} continues to be amenable to this particular case. The appendix below describes work in progress on this step.

\appendix

\section{Bounds from the Fourth Moment}\label{sec:applimiller}

\noindent By Jiahui (Stella) Li\footnote{\email{\textcolor{blue}{\href{mailto:stellali110204@gmail.com}{stellali110204@gmail.com}}}, Emma Willard School.} and Steven J. Miller\\ \

As remarked, there are several avenues for future research to obtain better estimates on the order of vanishing. One of course is to continue to try to improve the test functions, though this paper and earlier work show that simple choices are close to the optimal results. We may also study higher level densities (or equivalently centered moments). This paper is the first along this line, and concentrates on the 2-level density. We can of course continue and consider $n$-level densities and moments, and in those calculations important ingredients are good choices for fixed test functions, so that once again we can reduce the complexity of the optimization problem. Thus our results here can be fed in to higher $n$. There is a balancing act here; \emph{the larger $n$ give better bounds as the rank grow, but worse results for small ranks}\footnote{Often the analysis leads to bounds on the order of vanishing that exceed 100\%, and of course there are easier ways to get such estimates!}. Thus it is important to pursue both avenues, namely optimizing the test function and exploring larger results coming from larger $n$.



Calculations for bounds from higher level density are computationally difficult as they include multiple $n$-dimensional determinant integrals. A closely related statistic, the $n$\textsuperscript{th} centered moment, was used instead by Hughes and Miller \cite{HM}. The $n$\textsuperscript{th} centered moment replaces the $n \times n$ determinant expansions of the $n$-level density with a one-dimensional integral through a clever change of variable that reduces a $n$-dimensional integral involving Bessel functions and $n$ test functions to a related one-dimensional integral with just one Bessel function against a new test function.

Below we report on some work in progress that uses the results of this paper; see \cite{LiM} for details. As a first step, the results of \cite{HM} are generalized to allow the test function to be a product of $n$ distinct even Schwartz test functions, which is essential in order to be able to use the results of this paper as fixed test functions for some of the choices, leaving the other test functions free to be varied. Comparing the results with the bounds obtained using the $1$- and $2$-level densities, we see \emph{\textbf{significant}} improvements for bounds on modest rank, but worse results for bounds on small rank. Thus both methods, increasing $n$ and optimizing the test functions, have their utility. Below we report on ranks where the fourth moment bounds are better.



When calculating bounds using the $4\textsuperscript{th}$ centered moment, if we assume calculations hold for greater support, we find that the combination of a pair of the optimal 2-level test function found in this paper and a pair of the naive test function yields better results than just 4 naive test functions for higher ranks of SO(even) functions\footnote{We are optimistic that such larger support will be known soon, see \cite{C--}.}  The calculated bounds are illustrated in the tables below. Table \ref{ref:tablenaivethispaper} shows an application where the test functions found in this paper yield better results than the naive functions, while Tables \ref{ref:tableeven} and \ref{ref:tableodd} show the significant improvement in upper bounds for modest ranks.

\begin{table}[h]
\footnotesize
    \centering
    \begin{tabular}{|l|l|l|l|l|}
    \hline
        Order vanishing & $1$-level & $2$-level & $4$\textsuperscript{th} centered moment & $4$\textsuperscript{th} centered moment\textsuperscript{*} \\ \hline
        100 & 0.0086454 & 3.86172 $\cdot$ $10^{-5}$ & 4.06379 $\cdot$ $10^{-9}$ &  3.68655 $\cdot$ $10^{-9}$ \\ \hline
        200 & 0.0043227 & 9.55677 $\cdot$ $10^{-6}$ & 2.46396 $\cdot$ $10^{-10}$ & 2.18955 $\cdot$ $10^{-10}$ \\ \hline
        400 & 0.0021613 & 2.37719 $\cdot$ $10^{-6}$ & 1.51692 $\cdot$ $10^{-11}$ & 1.33434 $\cdot$ $10^{-11}$ \\ \hline
        800 & 0.0010806 & 5.92807 $\cdot$ $10^{-7}$ & 9.40972 $\cdot$ $10^{-13}$ & 8.23550  $\cdot$ $10^{-13}$ \\ \hline
    \end{tabular}
        \caption{\label{ref:tablenaivethispaper}\small
        Upper bounds for vanishing to order at least $r$ for SO(even) from various approaches. \\
        For the $1$-level, we used the optimal $1$-level bound from \cite{ILS}, the support is $(-2,2)$.\\
        For the $2$-level, we used the results of this paper, the support is $(-1,1)$.\\
        For the $4$\textsuperscript{th} moment, we used $4$ copies of the naive test function, the support is $(-1/2,1/2)$. \\
        For the $4$\textsuperscript{th} moment\textsuperscript{*}, we used $2$ copies of the results of this paper and $2$ copies of the naive test function, the support is $(-1/2,1/2)$.}
\end{table}

\begin{table}[h]
    \centering
    \begin{tabular}{|l|l|l|l|}
    \hline
        Order vanish & $1$-level & $2$-level & $4$\textsuperscript{th} centered moment\\ \hline
        6  & 0.144090 & 0.0157687 & 0.00853841 \\ \hline
        8  & 0.108067 & 0.0078843 & 0.000813368 \\ \hline
        10 & 0.086454 & 0.0047306 & 0.000186846 \\ \hline
        20 & 0.043227 & 0.0010512 & 4.49988 $\cdot$ $10^{-6}$ \\ \hline
    \end{tabular}
\caption{\label{ref:tableeven} Upper bounds for vanishing to order at least $r$ for SO(even). \\
The $1$-level uses the optimal $1$-level bound from \cite{ILS}, the support is $(-2,2)$.\\
The $2$-level uses the results of this paper, the support is $(-1,1)$.\\
For the $4$\textsuperscript{th} centered moment\textsuperscript{*} column, we use $4$ copies of the naive test functions with support $(-1/3,1/3)$.}
\end{table}

\begin{table}[h]
    \centering
    \begin{tabular}{|l|l|l|l|}
    \hline
        Order vanishing & $1$-level & $2$-level & $4$\textsuperscript{th} centered moment \\ \hline
        5  & 0.222908 & 0.0674429 & 0.06580440 \\ \hline
        7  & 0.159220 & 0.0299746 & 0.00221997 \\ \hline
        9  & 0.123838 & 0.0168607 & 0.00036405 \\ \hline
        19 & 0.058660 & 0.0033305 & 5.77156 $\cdot$ $10^{-6}$ \\ \hline
    \end{tabular}
        \caption{\label{ref:tableodd}
    Upper bounds of vanishing to order at least $r$ for SO(odd). \\
    The $1$-level uses the optimal $1$-level bound from \cite{ILS}, the support is $(-2,2)$.\\
    The $2$-level uses the results of this paper, the support is $(-1,1)$.\\
    For the $4$\textsuperscript{th} centered moment column, we use $4$ copies of the naive test functions with support $(-1/3,1/3)$.}
\end{table}

\newpage

\bibliography{biblio}

\end{document}